\documentclass{amsart}
\usepackage{amsfonts,amssymb,amsmath,amsthm}
\usepackage{url}
\usepackage{enumerate}

\usepackage[all,cmtip]{xy}
\usepackage{color}
\usepackage[T1]{fontenc}
\usepackage{enumerate}

\theoremstyle{plain}
\newtheorem{thm}{Theorem}[section]

\newtheorem{prop}[thm]{Proposition}
\newtheorem{cor}[thm]{Corollary}
\newtheorem*{assumption}{Assumption}
\theoremstyle{definition}
\newtheorem{defn}[thm]{Definition}
\newtheorem{exmp}[thm]{Example}
\newtheorem{rem}[thm]{Remark}
\numberwithin{equation}{section}

\DeclareMathOperator{\Aut}{Aut}
\DeclareMathOperator{\Hom}{Hom}
\DeclareMathOperator{\Homeo}{Homeo}
\DeclareMathOperator{\id}{id}
\DeclareMathOperator{\supp}{Supp}

\newcommand{\Z}{\mathbb{Z}}

\newcommand{\R}{\mathbb{R}}

\newcommand{\SkewRing}{A \star_\alpha G}

\begin{document}

\title[Outer Partial Actions and Partial Skew Group Rings]{Outer Partial Actions and Partial Skew Group Rings}

\author{Patrik Nystedt}
\address{University West, Department of Engineering Science, SE-46186
  Trollh\"{a}ttan, Sweden}
\email{Patrik.Nystedt@hv.se}

\author{Johan \"{O}inert}
\address{Centre for Mathematical Sciences, P.O. Box 118, Lund University, SE-22100 Lund, Sweden}
\email{Johan.Oinert@math.lth.se}

\thanks{The second author was partially supported by The Swedish Research Council (repatriation grant no. 2012-6113).}


\begin{abstract}
We extend the classicial notion of an outer action 
$\alpha$ of a group $G$ on a unital ring $A$
to the case when $\alpha$ is a partial action 
on ideals, all of which have local units.
We show that if $\alpha$ is an outer partial
action of an abelian group $G$,
then its associated partial skew group
ring $\SkewRing$ is simple if and only if 
$A$ is $G$-simple.
This result is applied to partial skew group rings associated with two different types of partial dynamical systems.
\end{abstract}

\maketitle

\pagestyle{headings}

\section{Introduction}

The notion of a partial action of a group on a C*-algebra, and the construction of its
associated crossed product C*-algebra, was introduced by R. Exel
\cite{Exel94,Exel98}
for partial actions of the integers and then 
extended by K. McClanahan \cite{Mcclanahan95} to partial actions of  discrete groups.
Since then, the theory of (twisted) partial actions on C*-algebras has
developed into a rich theory which has become an important tool
in the study of C*-algebras.
It is now known that several
important classes of C*-algebras
can be realized as crossed product C*-algebras by (twisted) partial actions,
e.g. AF-algebras \cite{Exel95}, Bunce-Deddens algebras \cite{Exel94BD}, Cuntz-Krieger algebras \cite{ExelLaca99}
and Cuntz-Li algebras \cite{BoavaExel13}.

In a purely algebraic context, partial skew group rings were 
introduced by M. Dokuchaev and R. Exel \cite{DokuchaevExel05}
as a generalization of classical skew group rings
and as an algebraic analogue of partial crossed product C*-algebras.
Compared to the abundance of results in the context of skew group rings or partial crossed product C*-algebras,
the theory of partial skew group rings is still underdeveloped.
In particular, apart from
the results in \cite{AvilaFerrero11,BeuterGoncalves,Goncalves13,GoncalvesOinertRoyer}, 
very little is known about the ideal structure and simplicity criteria
for partial skew group rings.

The primary goal of the present article is to establish a generalization
(see Theorem \ref{generalizationcrow}) of a result due to K. Crow \cite{crow2005}
(see Theorem \ref{crow})
concerning a connection between outer actions and simplicity
of unital skew group rings, 
to partial skew group rings which have local units.
The secondary goal is to apply this result to show generalizations
(see Theorem \ref{generalizationgoncalvesfinite} and Theorem \ref{generalizationgoncalves})
of recent results by D. Gon\c{c}alves \cite{Goncalves13}
concerning partial skew group rings associated with two different types of partial dynamical systems.

Before we describe these results, we first need to recall the following notions.
Let $G$ be a group with identity element $e$ and let $X$ be a set.
A \emph{partial action} $\alpha$
of $G$ on $X$ is a collection of subsets $\{ X_g \}_{g \in G}$
of $X$ and a collection of bijections
$\alpha_g : X_{g^{-1}} \to X_g$, for $g\in G$, such that
for all $g,h \in G$ and every $x \in X_{h^{-1}} \cap X_{(gh)^{-1}}$,
the following three relations hold:
\begin{itemize}
\item[(i)] $\alpha_e = \id_X$; 
\item[(ii)] $\alpha_g(X_{g^{-1}} \cap X_h) = X_g \cap X_{gh}$; 
\item[(iii)] $\alpha_g(\alpha_h(x)) = \alpha_{gh}(x)$.
\end{itemize}
It often happens that the set $X$ carries an additional structure.
By requiring that the subsets $\{X_g\}_{g\in G}$ and the bijections $\{\alpha_g\}_{g\in G}$
are compatible with the given structure on $X$, we may define a partial action of a certain type.
If $X$ is a topological space,
then we require that, for each $g\in G$, $X_g$ is an open
set and $\alpha_g$ is a homeomorphism.
If $X$ is a semigroup (ring, algebra), then we require that, for each $g \in G$, the
subset $X_g$ is an ideal of $X$ and the map 
$\alpha_g$ is a semigroup (ring, algebra) isomorphism.
A subset $I$ of $X$ is called
\emph{$G$-invariant} if, for each $g \in G$,
the inclusion $\alpha_g(I \cap D_{g^{-1}}) \subseteq I$ holds.
In case $X$ is a semigroup (ring, algebra),
we say that $X$ is \emph{$G$-simple}
if there is no $G$-invariant ideal of $X$ other than $X$ itself and $\{ 0 \}$ (which need not exist).
The action $\alpha$ is called {\it global} if
the equality $X_g = X$ holds for each $g\in G$.

As a preparation for K. Crow's result below, we shall
now recall a couple of important notions from the \emph{classical setting}, i.e.
when $X$ is a unital ring (algebra)
and $\alpha : G \ni g \mapsto \alpha_g \in \Aut(X)$
is a global action of $G$ on $X$.
If $g \in G$, then the map $\alpha_g$ is said to be
\emph{inner} if there is an invertible $a \in X$
such that the relation $\alpha_g(x) = a^{-1} x a$
holds for all $x \in X$.
The action $\alpha$ is said to be
\emph{outer} if the identity element $e$ is the only element of
$G$ that maps to an inner automorphism of $X$.

\begin{thm}[Crow \cite{crow2005}]\label{crow}
If $\alpha : G \rightarrow \Aut(A)$ is an outer action (in the classical sense)
of an abelian group $G$ on a unital ring $A$,
then the associated skew group ring $A *_{\alpha} G$
is simple if and only if $A$ is $G$-simple.
\end{thm}

To describe our generalization of Theorem \ref{crow}
and its applications, 
we first need to answer the following question: 
\begin{center}
{\it What should it mean for a partial action of a group on a ring to be outer?}
\end{center}
As far as we know, this question has not
previously been analysed in the literature,
neither in the C*-algebra context, nor in the 
purely algebraical setting.
The starting point for our investigations
is the observation that many of the concepts 
concerning partial actions on rings 
are formulated by using only the operation of multiplication, 
and thus forgetting the additive structure.
In other words, we are working in the
multiplicative semigroup of a ring.

In Section \ref{sectionsemigroups}, we therefore
begin our explorations in a general semigroup $S$.
In addition, since we want to establish a 
non-unital version of Theorem \ref{crow},
we also have to decide on what it should mean for isomorphisms
$\alpha : I \rightarrow J$
of ideals $I$ and $J$ in $S$ to be outer, {\it locally at idempotents} $u \in I \cap J$. 
To motivate the approach taken later,
let us briefly describe the train of reasoning
that lead us to the formal definition. 
The restricted map 
$\alpha|_{uSu} : uSu \rightarrow \alpha(u) S \alpha(u)$
is also an isomorphism of semigroups.
So by mimicking the global case, the map
$\alpha|_{uSu}$ should be called ''inner'' if there are $a,b \in S$
such that $\alpha|_{uSu}(x) = bxa$ holds for all $x \in uSu$.
However, for such a definition to make sense, 
we need to assume that $a \in u S \alpha(u)$
and $b \in \alpha(u) S u$.
From the fact that $\alpha(u) = \alpha|_{uSu}(u) = bua = ba$
we get that $ba = \alpha(u)$.
Also, the inverse of $\alpha|_{uSu}$ should be defined
by the ''reversed'' map $a(\cdot)b$ from which we get that $ab = u$. 
Therefore, if such $a$ and $b$ exist, we say that $\alpha$
is inner at $u$; otherwise $\alpha$ is called outer at $u$
(see Definition \ref{cornerinner} for more details).

In Section \ref{sectiongradedrings},
we recall a result (see Theorem \ref{gather}) from \cite{GroupoidGradedRings}
by the authors of the present article concerning simplicity of
group graded rings which we will need in the subsequent section
for application to partial skew group rings,
which, in a natural way, are group graded rings.

In Section \ref{sectionpartialactions},
we use the definition of outer actions
in semigroups from Section \ref{sectionsemigroups} to define
outer partial actions $\alpha_g : D_{g^{-1}} \rightarrow D_g$ of a group $G$
on a ring $A$ in the following way (see Definition \ref{defnouter} for more details).
Consider $A$ as a semigroup with respect to multiplication.
If $g \in G$, then we say that $\alpha_g$
is inner (outer) at an idempotent $u \in A$
if it is inner (outer) at $u$ in the sense defined above.
Furthermore, we say that $\alpha$ is {\it outer} (or \emph{outer at $u$})
if there is a non-zero idempotent $u \in A$ such that
for each non-identity $g \in G$,
the map $\alpha_g$ is outer at $u$.
In the classical setting, i.e. when 
$A$ is unital and $\alpha$ is a global
action of $G$ on $A$, our definition
of outerness coincides with the classical
definition of outerness described above
(see Remark \ref{coincides}).
At the end of Section \ref{sectionpartialactions},
we show, with the aid of the result in Section \ref{sectiongradedrings}, 
the following generalization of Theorem \ref{crow}.

\begin{thm}\label{generalizationcrow}
If $\alpha_g : D_{g^{-1}} \rightarrow D_g$, for $g \in G$,
is an outer partial action of an abelian group $G$ on 
a ring $A$ such that
$D_g$, for each $g \in G$,
has local units, then the associated partial 
skew group ring $\SkewRing$ is simple if and only if $A$ is $G$-simple.
\end{thm}

In Section \ref{Sec:SetDynamics} and Section \ref{Sec:TopDynamics},
we show that Theorem \ref{generalizationcrow} effectively can be
applied to set dynamics respectively topological dynamics.
To be more precise, let us recall the following notions
for a partial action $\alpha$ of a group $G$ on a set
(topological space) $X$.
If for each non-identity $g \in G$, there is some $x \in X_{g^{-1}}$
such that $\alpha_g(x) \neq x$, then $\alpha$
is said to be \emph{faithful}.
If for each non-identity $g \in G$, the set of $x \in X_{g^{-1}}$ 
which satisfy $\alpha_g(x)=x$, is the empty set
(has empty interior), 
then $\alpha$ is called (topologically) \emph{free}.
Clearly, freeness implies topological freeness.
If $X$ and $\emptyset$ are the only $G$-invariant (closed) subsets of $X$,
then $\alpha$ is said to be (topologically) \emph{minimal}.

In the set dynamical case, we are given a partial action $\alpha$ of a group $G$
on a (non-empty) set $X$ and consider the partial skew group ring $F_0(X,B) \star_\alpha G$.
Here $F_0(X,B)$ denotes the algebra of finitely supported functions $X \to B$,
where $B$ is a simple associative
ring which has local units.

\begin{thm}\label{generalizationgoncalvesfinite}
If $G$ is abelian,
then the following three assertions are equivalent:
\begin{enumerate}[{\rm (i)}]
	\item $F_0(X,B) \star_\alpha G$ is simple;
	\item $\theta$ is minimal and free;
	\item $\theta$ is minimal and faithful.
\end{enumerate}
\end{thm}

In the topological dynamical case, we are given a partial action $\alpha$ of a group $G$
on a compact Hausdorff
space $X$ such that each $X_g$, for $g\in G$,
is clopen. 
Note that if $G$ is a countable discrete group, 
then these partial actions are exactly the ones for which
the enveloping space is Hausdorff (see \cite[Proposition 3.1]{ExelGiordanoGoncalves11}).
We then consider the partial skew group ring $C_E(X,B) \star_\alpha G$.
Here $B$ denotes a simple associative
topological real algebra
which has a set $E$ of local units. (Some additional assumptions are made 
on $B$, see Section \ref{Sec:TopDynamics}.)
The algebra $C_E(X,B)$ is the directed union of the ''local'' algebras
$C(X, \epsilon B \epsilon) = \{ \text{continuous } f : X \rightarrow \epsilon B \epsilon \}$
where $\epsilon$ runs over all elements in $E$.

\begin{thm}\label{generalizationgoncalves}
If $G$ is abelian, $X$ is compact Hausdorff and each 
$X_g$, for $g \in G$, is clopen, then the
following three assertions are equivalent:
\begin{enumerate}[{\rm (i)}]
	\item $C_E(X,B) \star_\alpha G$ is simple;
	\item $\theta$ is topologically minimal and topologically free;
	\item $\theta$ is topologically minimal and faithful.
\end{enumerate}
\end{thm}

Note that Theorem \ref{generalizationgoncalvesfinite} and Theorem \ref{generalizationgoncalves}
generalize recent results by D. Gon\c{c}alves \cite{Goncalves13}
to also include cases when the coefficients are taken
from {\it non-commutative} rings which have local units.

\section{Outer Actions of Ideals in Semigroups}\label{sectionsemigroups}

In this section, we introduce the 
concepts of innerness and outerness of homomorphisms
of ideals in semigroups at idempotents (see Definition \ref{cornerinner}).
We also show that the innerness is preserved
by the classical partial order on
the idempotents in the semigroup
(see Proposition \ref{proppartialorder}).
We begin by fixing some notation.

Throughout this section, $S$ denotes a semigroup.
By this we mean that $S$ is a non-empty
set equipped with an associative
binary operation $S \times S \ni (x,y) \mapsto xy \in S$, which is referred to as the \emph{multiplication} of the semigroup.
For subsets $I$ and $J$ of $S$ we let 
$IJ$ denote the set of all products of the
form $xy$ for $x \in I$ and $y \in J$.
A non-empty subset $I$ of $S$ is called a subsemigroup
(left ideal, right ideal, ideal) of $S$
if $II \subseteq I$ ($SI \subseteq I$, $IS \subseteq I$,
$SI \cup IS \subseteq I$). 
If $T$ is another semigroup, then a 
map $\alpha : S \rightarrow T$ is a 
homomorphism of semigroups if it 
respects the multiplication in $S$ and $T$.
Suppose that $I$ and $J$ are right ideals of $S$. 
Then a map $\alpha : I \rightarrow J$ is called
a homomorphism of right ideals if $\alpha(xy) = \alpha(x)y$,
for $x \in I$ and $y \in S$.
We let $\Hom_S(I,J)$ denote the set of all
homomorphisms $I \rightarrow J$ of right ideals.
The concept of a homomorphism
of (left) ideals is defined analogously.

The first two propositions below have already appeared
in the context of ideals in rings
(see e.g. Proposition (21.6) and Proposition (21.20) in \cite{lam91}),
except for the last part of the first proposition.
However, we were not able to find
an appropriate reference for the case of semigroups.
The proofs are a close adaptation 
to semigroups of the proofs given in loc. cit.
and we include them for the convenience of the reader.

\begin{prop}\label{Hom}
Let $u$, $v$ and $w$ be idempotents in $S$
and suppose that $I$ is a right ideal of $S$.
Then the map of sets $\lambda : \Hom_S(uS,I) \rightarrow Iu$,
defined by $\lambda(\beta) = \beta(u)$,
for $\beta \in \Hom_S(uS,I)$, is a bijection.
In particular, if we put $I = vS$, 
then the corresponding map 
$\lambda_{v,u} : \Hom_S(uS,vS) \rightarrow vSu$
is a bijection. 
Moreover, if $\beta \in \Hom_S(uS,vS)$ and $\beta' \in \Hom_S(vS,wS)$, then 
$\lambda_{w,v}(\beta') \lambda_{v,u}(\beta) = (\lambda_{w,u})(\beta' \circ \beta)$.
\end{prop}

\begin{proof}
First we show that $\lambda$ is well-defined.
Suppose that $\beta : uS \rightarrow I$
is a right ideal homomorphism.
Then $\lambda(\beta) = \beta(u) = \beta(u^2) = 
\beta(u)u \in Iu$.
Next, we show that $\lambda$ is injective.
Suppose that $\beta$ and $\beta'$ are 
right ideal homomorphisms $uS \rightarrow I$
such that $\lambda(\beta) = \lambda(\beta')$.
Take $s \in S$. Then
$\beta(us) = \beta(u)s = \lambda(\beta)s =
\lambda(\beta')s = \beta'(u)s = \beta'(us)$.
Therefore $\beta = \beta'$.
Finally, we show that $\lambda$ is surjective.
Take $iu \in Iu$, where $i \in I$.
Define $\beta_{iu} \in \Hom_S(uS,I)$ by
$\beta_{iu}(us) = ius$, for $s \in S$.
We claim that $\beta_{iu}$ is well-defined.
If we assume that the claim holds, then
$\lambda(\beta_{iu}) = \beta_{iu}(u) = \beta_{iu}(uu) = 
iuu = iu$ and thus $\lambda$ is surjective.
Now we show the claim.
Suppose that $us = us'$ for some $s,s' \in S$.
Then $\beta_{iu}(us) = ius = ius' = \beta_{iu}(us')$.
The second part follows immediately from the first part.
Now we show the last part of the proof.
Take $\beta \in \Hom_S(uS,vS)$ and $\beta' \in \Hom_S(vS,wS)$.
Then $\lambda_{w,v}(\beta') \lambda_{v,u}(\beta) =
\beta'(v) \beta(u) = \beta'(v \beta(u)) = \beta'(\beta(u)) =
\lambda_{w,u}(\beta' \circ \beta)$.
\end{proof}

\begin{prop}
If $u$ and $v$ are idempotents of $S$,
then the following four assertions are equivalent:
\begin{itemize}
\item[{\rm (a)}] $uS \cong vS$ as right $S$-ideals;

\item[{\rm (b)}] $Su \cong Sv$ as left $S$-ideals;

\item[{\rm (c)}] There exist $a \in uSv$ and $b \in vSu$
such that $ab = u$ and $ba = v$;

\item[{\rm (d)}] There exist $a,b \in S$
such that $ab = u$ and $ba = v$. 
\end{itemize}
\end{prop}

\begin{proof}
By left-right symmetry it is enough to show 
(a)$\Rightarrow$(c)$\Rightarrow$(d)$\Rightarrow$(a).

(a)$\Rightarrow$(c):
Let $\beta : uS \rightarrow vS$ be an isomorphism of right ideals.
Put $a = \lambda_{v,u}(\beta)$ and $b = \lambda_{u,v}(\beta^{-1})$.
Then, by the last part of Proposition \ref{Hom}, we get
$u = \lambda_{u,u}(\id_{uS}) = \lambda_{u,u}(\beta^{-1} \circ \beta) =
\lambda_{u,v}(\beta^{-1}) \lambda_{v,u}(\beta) = ba$ and
$v = \lambda_{v,v}(\id_{vS}) = \lambda_{v,v}(\beta \circ \beta^{-1}) =
\lambda_{v,u}(\beta) \lambda_{u,v}(\beta^{-1}) = ab$.

(c)$\Rightarrow$(d): Trivial.

(d)$\Rightarrow$(a): Suppose that there are
$a,b \in S$ such that $ab = u$ and $ba = v$.
Define $\beta : uS \rightarrow vS$
and $\gamma : vS \rightarrow uS$
by the relations $\beta(x) = bx$, for $x \in uS$,
and $\gamma(y) = ay$, for $y \in vS$, respectively.
Since $bx = b u x = b ab x = vbx$, for $x \in uS$,
and $ay = a vy = a ba y = u a y$, for $y \in vS$, 
it follows that $\beta$ and $\gamma$ are well-defined homomorphisms of right ideals.
Now we show that $\gamma \circ \beta = \id_{uS}$
and
$\beta \circ \gamma = \id_{vS}$.
Take $x \in uS$ and $y \in vS$.
Then $(\gamma \circ \beta) (x) = \gamma(bx) = abx = ux = x$
and $(\beta \circ \gamma) (y) = \beta(ay) = bay = vy = y$. 
\end{proof}

\begin{defn}\label{equivalent}
Let $u$ and $v$ be idempotents of $S$.
We say that $u$ and $v$ are {\it equivalent},
and denote this by $u \sim v$,
if $u$ and $v$ satisfy any (and hence all) of the 
equivalent conditions (a)-(d) above.
\end{defn}

\begin{defn}\label{cornerinner} 
Suppose that $I$ and $J$ are ideals of $S$
and $\alpha : I \rightarrow J$
is a semigroup homomorphism.
Let $u$ be an idempotent of $S$.
We say that $\alpha$ is {\it inner at} $u$
if $u \in I$ and $u \sim \alpha(u)$ where this equivalence is defined by
an isomorphism $\beta : uS \rightarrow \alpha(u)S$
of right $S$-ideals such that 
$\alpha(x) = \beta(u) x \beta^{-1}(\alpha(u))$
for all $x \in uSu$.
We say that $\alpha$ is {\it outer at} $u$
if $\alpha$ is not inner at $u$.
We say that $\alpha$ is {\it strongly outer}
if it is outer at {\it all} non-zero idempotents of $S$.
\end{defn}

\begin{rem}
Suppose that $I$ and $J$ are ideals of $S$
and that $\alpha : I \rightarrow J$
is a semigroup homomorphism
which is inner at an idempotent $u$ of $I$.
\begin{itemize}
	\item[(a)] Although we in the above definition
only assume that $\alpha : I \rightarrow J$ 
is a semigroup homomorphism, the restricted
map $\alpha|_{uSu} : uSu \rightarrow \alpha(u) S \alpha(u)$
is always an {\it isomorphism} of semigroups. 
In fact, if we put $a = \beta^{-1}(\alpha(u))$ and $b = \beta(u)$,
then $ba = \alpha(u)$ and $ab = u$ and
$\alpha(x) = bxa$ for all $x \in uSu$.
It is now clear that 
$\alpha|_{uSu}^{-1} : \beta(u) S \beta(u) \rightarrow uSu$
is defined by $\alpha|_{uSu}^{-1}(x) = axb$
for all $x \in \beta(u) S \beta(u)$.

\item[(b)] It follows that $u \in I \cap J$, since
$u = ab = a \alpha(u) b \in a J b \subseteq J$.

\item[(c)] If $S$ is a monoid 
and we let $u$ be the identity element of $S$, 
then $\alpha : S \to S$ is inner at $u$ precisely
when it is inner in the classical case, i.e. if
there is an invertible $y \in S$ such that
$\alpha(x) = y x y^{-1}$ for all $x \in S$.
In particular, by (a), this forces $\alpha$
to be a semigroup automorphism of $S$.
\end{itemize}
\end{rem}

\begin{defn}\label{partialorder}
Recall that the idempotents of $S$ 
can be partially ordered by saying that
$v \leq u$ if $uv = vu = v$.
An idempotent is called {\it minimal}
if it is minimal with respect to $\leq$.
\end{defn}

\begin{prop}\label{proppartialorder}
Suppose that $I$ and $J$ are ideals of $S$
and that $\alpha : I \rightarrow J$
is a semigroup homomorphism
which is inner at an idempotent $u$ of $I$.
If $v$ is another idempotent of $I$
with $v \leq u$, 
then $\alpha$ is inner at $v$.
\end{prop}

\begin{proof}
Suppose that there is an isomorphism $\beta : uS \rightarrow \alpha(u)S$
of right ideals such that 
$\alpha(x) = \beta(u) x \beta^{-1}(\alpha(u))$
for all $x \in uSu$.
Put $b = \beta(u)$ and $a = \beta^{-1}(\alpha(u))$.
Then $ab = u$ and $ba = \alpha(u)$
and there are some $d,d' \in S$ such that 
$a = u d \alpha(u)$ and $b = \alpha(u) d' u$.

Consider the elements $a' = v d \alpha(v)$ and $b' = \alpha(v) d' v$.
Then
$a \alpha(x) b = a (bxa)b = u x u = x$ holds for any $x \in uSu$.
In particular, for $x=v$ this yields $a \alpha(v) b = v$
and hence
\begin{align*}
a'b' &= (v d \alpha(v)) (\alpha(v) d' v) = 
v d \alpha(v) d' v = 
vu d \alpha(u) \alpha(v) \alpha(u) d' uv \\
&= v(u d \alpha(u)) \alpha(v) (\alpha(u) d' u)v = vvv = v.
\end{align*}
Moreover,  
$bva = \alpha(v)$ and hence 
\begin{align*}
	b'a' &= (\alpha(v) d' v) (v d \alpha(v)) = 
\alpha(v) d' v d \alpha(v) 
= \alpha(v) \alpha(u) d' u v u d \alpha(u) \alpha(v) \\
&=\alpha(v) (\alpha(u) d' u ) v (u d \alpha(u)) \alpha(v) =
\alpha(v) \alpha(v) \alpha(v) = \alpha(v).
\end{align*}
Take $x \in vSv \subseteq uSu$.
There is some $z\in S$ such that $x=vzv$. Hence, $\alpha(x)=\alpha(vzv)=
\alpha(v)\alpha(zv)=\alpha(vz)\alpha(v)$.
This shows that $\alpha(x) = \alpha(v) \alpha(x) \alpha(v)$.
Then
\begin{align*}
	\alpha(x) &= (\alpha(u) d' u) x (u d \alpha(u))
= (\alpha(u) d' u) v x v (u d \alpha(u))
= \alpha(u) d' v x v d \alpha(u)\\
&= \alpha(v) ( \alpha(u) d' v x v d \alpha(u) ) \alpha(v)
= (\alpha(v) d' v) x (v d \alpha(v) ) = b' x a'.
\end{align*}
This shows that $\alpha$ is inner at $v$.
\end{proof}

\begin{rem}\label{rem:innerlocalglobal}
The conclusion of Proposition \ref{proppartialorder}
does not hold, in general, if $v \leq u$ is replaced by $u \leq v$.
In particular, local innerness can not always be lifted to global innerness.
To be more precise, suppose that $I$ and $J$ are ideals of $S$
and that $\alpha : I \rightarrow J$
is a semigroup homomorphism.
If $u,v \in S$ are idempotents such 
that $v \leq u$ and $\alpha$ is inner at $v$,
then this does not in general imply that 
$\alpha$ is inner at $u$.
In fact, let $S = I = J$ denote the multiplicative semigroup
of functions from $\{ 1,2,3 \}$ to a field $K$.
Let $u,v \in S$ be defined by  
$u(1)=u(2)=u(3)=1_K$ respectively $v(1)=1_K$ and $v(2)=v(3)=0$.
Then $v \leq u$.
If we define $\alpha : S \rightarrow S$ by
$\alpha(f)(1)=f(1)$, $\alpha(f)(2)=f(3)$ and $\alpha(f)(3)=f(2)$,
for all $f \in S$,
then it is easy to see that 
$\alpha|_{vSv} = \id_{vSv}$.
Clearly, $\alpha$ is inner at $v$,
but
outer at $u$.
\end{rem}

\begin{defn}\label{completeminimal}
We say that a set $E$ of minimal non-zero idempotents of $S$
is a {\it complete set of minimal idempotents} if
for each non-zero idempotent $u \in S$, there
is $v \in E$ such that $v \leq u$. 
\end{defn}

\begin{cor}
Suppose that there is a complete set $E$ of minimal 
idempotents of $S$.
Let $I$ and $J$ be ideals of $S$
and suppose that $\alpha : I \rightarrow J$
is a semigroup homomorphism.
Then $\alpha$ is strongly outer if and only if 
it is outer at each $u \in E$.
\end{cor}

\begin{proof}
This follows immediately from Proposition \ref{proppartialorder}
and Definition \ref{completeminimal}. 
\end{proof}

\begin{rem}
Innerness of ring automorphisms at idempotents
(however not in the generality of semigroup homomorphisms of ideals) 
have been considered by J. Haefner and A. del Rio 
in \cite[Definition 1.2 on p. 38]{haefnerdelrio}. 
\end{rem}

\section{Simple Group Graded Rings}\label{sectiongradedrings}

In this section, we recall a result 
(see Theorem \ref{gather}) from \cite{GroupoidGradedRings}
by the authors of the present article concerning simple
group graded rings which we will need in the sequel.
We begin by fixing some notation.

Let $R$ denote a ring which is associative but not necessarily unital.
If $R$ is unital, then we let $1_R$
denote its multiplicative identity element.
By an \emph{ideal} of $R$ we always mean a two-sided 
ideal of $R$.
The \emph{center} of $R$, denoted by $Z(R)$,
is the set of elements $x \in R$ with the property
that $xy = yx$ holds for each $y \in R$.
Recall from \cite{anh87} that $R$ is said to have \emph{local units} if
there exists a set $E$ of idempotents of $R$ such that,
for every finite subset $X$ of $R$,
there exists an $f\in E$ such that $X \subseteq fRf$. 
From this it follows that
$x = f x = x f$ holds for each $x\in X$.

Let $G$ denote a group with identity element $e$.
Recall that $R$ is said to be 
\emph{graded} (by $G$), if there for each $g \in G$
is an additive subgroup $R_g$ of $R$
such that $R = \oplus_{g \in G} R_g$
and the inclusion $R_g R_h \subseteq R_{gh}$ holds
for all $g,h \in G$.
Take $r \in R$.
There are unique
$r_g \in R_g$, for $g \in G$,
such that all but finitely many of them
are zero and $r = \sum_{g \in G} r_g$.
We let the \emph{support} of $r$,
denoted by $\supp(r)$, be the set of 
$g \in G$ such that $r_g \neq 0$.
The element $r$ is called
\emph{homogeneous} if $|\supp(r)| \leq 1$.
If $r \in R_g \setminus \{ 0 \}$, for some $g \in G$,
then we write $\deg(r)=g$.
An additive subgroup $A$ of $R$,
is called \emph{graded}
if $A = \oplus_{g \in G} (A \cap R_g)$ holds.
The ring $R$ is said to be \emph{graded simple} 
if $R$ and $\{0\}$ are its only graded ideals.
Clearly, graded simplicity is
a necessary condition for simplicity.

\begin{thm}\label{gather}
If $R$ is a ring graded by an abelian group $G$
and $R_e$ contains a non-zero idempotent $u$,
then $R$ is simple if and only if 
it is graded simple and $Z(uRu)$ is a field.
\end{thm}

\begin{proof}
This follows from a more general result, by the authors
of the present article, concerning simplicity of 
semigroup graded rings (see \cite[Theorem 2]{GroupoidGradedRings}).
For the convenience of the reader, we now give a direct proof.
The ''only if' statement is straightforward.
Now we show the ''if'' statement.
Let $I$ be a non-zero ideal of $R$.
Take $r \in I \setminus \{0\}$ such that $|\supp(r)|$ is minimal.
Choose some $g\in G$ such that $r_g$ is non-zero.

Since $R$ is graded simple, there are
homogeneous $s_i,t_i \in R$, for $i = \{1,\ldots,n\}$,
such that $\sum_{i=1}^n s_i r_g t_i = u$.
In particular, there is $j \in \{ 1,\ldots,n \}$
such that $s_j r_g t_i \in R_e \setminus \{ 0 \}$.
By replacing $r$ with $s_j r t_j$, we can 
from now on assume that $r_e$ is non-zero.

Next we show that we may suppose that $r_e = u$.
Put
\begin{displaymath}
	J = \{ s_e \mid s \in R r R, \ \supp(s) \subseteq \supp(r) \}.
\end{displaymath}
Then $J$ is a non-zero ideal of $R_e$ and hence
$R J R$ is a non-zero graded ideal of $R$.
By graded simplicity of $R$ we get that 
there are $s^{(i)} \in R r R$
and $v_i,w_i \in R$, for $i \in \{1,\ldots,n\}$,
such that $\supp( s^{(i)} ) \subseteq \supp(r)$ and
$u = \sum_{i=1}^n v_i s^{(i)}_e w_i$.
From the last equality it follows that we may suppose that
$\deg(v_i) \deg(w_i) = e$
for all $i$ such that $v_i s^{(i)}_e w_i \neq 0$.
Put $s =  \sum_{i=1}^n v_i s^{(i)} w_i$.
Then $s \in I$ and since 
$\supp(s^{(i)}) \subseteq \supp(r)$ for all $i$
and $G$ is abelian,
we get that $\supp(s) \subseteq \supp(r)$. 
Therefore, $u = \sum_{i=1}^n v_i s^{(i)}_e w_i = s_e \in J$.

Finally we show that $I=R$.
Take $h \in G$ and $t \in u R_h u$.
Since $r_e=u$ and $G$ is abelian,
we get that $|\supp(rt - tr)| < |\supp(r)|$.
By the assumption that $|\supp(r)|$ is minimal, and the fact that $rt - tr \in I$,
we get that
$\supp(rt - tr) = \emptyset$ and hence
that $rt - tr = 0$.
Since $h \in G$ was arbitrarily chosen, we get that $r \in Z(uRu) \cap I$.
Using that $Z(uRu)$ is a field, we get that $u \in I$.
Therefore, since $R$ is graded simple, we get that
$R = RuR \subseteq I$.
\end{proof}

\section{Partial Actions and Partial Skew Group Rings}\label{sectionpartialactions}

In this section, we introduce
outer partial actions of groups 
on rings (see Definition \ref{defnouter})
and we prove the main result of this 
article concerning simplicity
of partial skew group rings (see Theorem \ref{generalizationcrow}).

\begin{assumption}
Throughout this section, 
$\alpha$ will denote a partial action of a group $G$
on a ring $A$, and the corresponding ideals of $A$ 
are denoted by $D_g$, for $g \in G$.
\end{assumption}

\begin{defn}
The \emph{partial skew group ring} 
$\SkewRing$ is defined as the set of all finite formal sums 
$\sum_{g\in G} a_g \delta_g$, where for each $g\in G$, 
$a_g \in D_g$ and $\delta_g$ is a symbol.
Addition is defined in the obvious way and multiplication 
is defined as the linear extension of the rule
$(a_g \delta_g)(b_h \delta_h)=\alpha_g(\alpha_{g^{-1}}(a_g)b_h) \delta_{gh}$
for $g,h\in G$, $a_g \in D_g$ and $b_h \in D_h$.
Clearly, each classical skew group ring
(see e.g. \cite{crow2005,FisherMontgomery78,oinertarxiv11}) is 
a partial skew group ring where $D_g = A$ for all $g \in G$.
\end{defn}

\begin{rem}
A partial skew group ring $A \star_\alpha G$ need not 
in general be associative (see \cite[Example 3.5]{DokuchaevExel05}).
However, if each $D_g$, for $g \in G$, has  
local units, then, in particular, each $D_g$, for $g \in G$,
is an idempotent ring, i.e. $D_g^2 = D_g$, which
by \cite[Corollary 3.2]{DokuchaevExel05},
ensures that $\SkewRing$ is associative.
In that case, the set $E \delta_e = \{f\delta_e \mid f\in E\}$ 
is a set of local units for $\SkewRing$,
if $E$ is a set of local units for $A$.
\end{rem}

\begin{defn}
If there does not exist
any non-identity $g \in G$ such that $D_g \cap D_{g^{-1}}$ is non-zero and
$\alpha_g|_{D_g \cap D_{g^{-1}}} = \id_{D_g \cap D_{g^{-1}}}$,
then $\alpha$ is said to be \emph{injective}.
\end{defn}

The next result extends a well-known result for group actions on rings
(see e.g. \cite{oinertarxiv11}),
to the case of partial actions.

\begin{prop}\label{simpleimpliesinjective}
If the partial
skew group ring $\SkewRing$ is simple,
then $\alpha$ is injective.
\end{prop}

\begin{proof}
Suppose that $\alpha$ is not injective.
Then there is a non-identity $g \in G$
such that $D_g \cap D_{g^{-1}} \neq \{0\}$
and $\alpha_g|_{D_g \cap D_{g^{-1}}} = \id_{D_g \cap D_{g^{-1}}}$.
Take a non-zero element $i \in D_g \cap D_{g^{-1}}$.
Let $J$ be the ideal of $\SkewRing$
generated by the element $i \delta_e - i \delta_g$.
It is clear that $J$ is non-zero and strictly
contained in $\SkewRing$. 
Therefore, $\SkewRing$ is not simple.
\end{proof}

\begin{rem}
Note that $\SkewRing$ need not be associative
for Proposition \ref{simpleimpliesinjective} to hold.
\end{rem}

\begin{rem}
It is easy to check that if we put
$(\SkewRing)_g = D_g \delta_g$, for $g \in G$,
then this defines a gradation on the ring $\SkewRing$.
In the sequel, whenever we speak of \emph{graded} or \emph{graded simple} it will
be with respect to this gradation.
\end{rem}

\begin{prop}\label{SimpleImpliesGSimple}
If each $D_g$, for $g \in G$, has local units, 
then $\SkewRing$ is graded simple
if and only if $A$ is $G$-simple.
\end{prop}

\begin{proof}
We begin by showing the ''only if'' statement. Suppose that $\SkewRing$ is graded simple.
Let $I$ be a non-zero $G$-invariant ideal of $A$.
Define $I \star_\alpha G$ to be the set 
of all finite sums of the form $\sum_{g\in G} a_g \delta_g$, where $a_g\in I \cap D_g$, for $g\in G$.
Note that $I \star_\alpha G$ is a non-zero two-sided graded ideal of $\SkewRing$.
Hence, $I \star_\alpha G = \SkewRing$. In particular, 
$A \delta_e \subseteq I \star_\alpha G$ which shows that $I \subseteq A \subseteq I$.
We conclude that $I=A$. Thus, $A$ is $G$-simple.

Now we show the ''if'' statement. Suppose that $A$ is $G$-simple.
Let $J$ be a non-zero graded ideal of $\SkewRing$.
We claim that $J_e = J \cap A$ is a non-zero $G$-invariant ideal of $A$. 
If we assume that the claim holds, then 
$A = J_e = A \cap J \subseteq J$ from which it follows that $J = \SkewRing$.
Now we show the claim.
First we show that $J_e$ is non-zero.
Since $J$ is non-zero, there is $g \in G$
and a non-zero $a_g \in D_g$ with $a_g \delta_g \in J$.
Let $b_{g^{-1}} \in D_{g^{-1}}$ be a local 
unit for $\alpha_{g^{-1}}(a_g)$.
Then
\begin{displaymath}
	J \ni a_g \delta_g b_{g^{-1}} \delta_{g^{-1}} =
\alpha_g( \alpha_{g^{-1}}(a_g) b_{g^{-1}} ) \delta_e =
\alpha_g( \alpha_{g^{-1}}(a_g)) \delta_e =
a_g \delta_e
\end{displaymath}
which is non-zero.
Now we show that $J_e$ is $G$-invariant.
Take $g\in G$ and $a \in J_e \cap D_{g^{-1}}$.  
Let $c_g\in D_g$ be such that $\alpha_{g^{-1}}(c_g)$ is a local unit for $a$.
Then
$\alpha_g(a)\delta_e=\alpha_g(\alpha_{g^{-1}}(c_g) a )\delta_e 
= c_g \delta_g a \delta_{g^{-1}} \in J$.
\end{proof}

\begin{rem}
Note that, even if there is some $g\in G$ such that $D_g$ does not have local units,
the first half of the above proposition still holds, as long as $\SkewRing$ is associative.
That is, graded simplicity of $\SkewRing$ implies $G$-simplicity of $A$.
In particular, simplicity of $\SkewRing$ implies $G$-simplicity of $A$.
\end{rem}

\begin{defn}\label{defnouter}
Consider $A$ as a semigroup with 
respect to multiplication.
If $g \in G$, then we say that $\alpha_g$ is inner at 
an idempotent $u \in A$
if it is \emph{inner at $u$} in the sense of Definition \ref{cornerinner}.
Moreover, we say that $\alpha$ is {\it outer} (or \emph{outer at $u$})
if there is a non-zero idempotent $u \in A$ such that
for each non-identity $g \in G$,
the map $\alpha_g$ is outer at $u$ in the sense
of Definition \ref{cornerinner}.
We say that $\alpha$ is {\it strongly outer}
if for every non-identity $g \in G$,
the map $\alpha_g$ is strongly outer
in the sense
of Definition \ref{cornerinner}.
\end{defn}

\begin{rem}\label{coincides}
Suppose that $A$ is unital and
that $\alpha : G \rightarrow \Aut(A)$
is a global action.
Then $\alpha$ is outer in the classical sense
if and only if it is outer in our sense, i.e. in the sense of Definition \ref{defnouter}.
This follows from Proposition \ref{proppartialorder}
and the fact that $u \leq 1$ holds for any idempotent $u$ of $A$.
\end{rem}

Suppose that $\beta$ is a global action of a group $G$ on a ring $B$
and that $A$ is an ideal of $B$.
If we, for each $g\in G$, define $D_g = A \cap \beta_g(A)$
and $\alpha_g(x)=\beta_g(x)$ for $x\in D_{g^{-1}}$,
then it is easily verified that $\alpha$ is a partial action
of $G$ on $A$.
In this situation, $\alpha$ is referred to as a \emph{restriction} of $\beta$,
and
$\beta$ is referred to as a \emph{globalization} of $\alpha$.
(See e.g. \cite{DokuchaevExel05,DokuchaevRio05}.)

\begin{prop}\label{OuterGlobalization}
Let $\alpha$ be a partial action of a group $G$ on a ring $A$
and suppose that $\alpha$ has a globalization $\beta$ (on a ring $B$).
The following two assertions hold:
\begin{itemize}
	\item[(a)] If $u$ is a non-zero idempotent of $A$, then, for $g\in G$,
	the map $\alpha_g$ is inner at $u$ if and only if $\beta_g$ is inner at $u$;
	\item[(b)] If $\alpha$ is outer, then $\beta$ is outer.
	Moreover, if $B$ is unital, then $\beta$ is outer in the classical sense.
\end{itemize}
\end{prop}

\begin{proof}
(a): We first show the contrapositive of the ''if'' statement.
Suppose that $\beta_g$ is inner at $u$.
There are elements $a \in u B \beta_g(u)$ and $b \in \beta_g(u) B u$, satisfying $ab=u$ and $ba=\beta_g(u)$,
such that $\beta_g(x)=bxa$ holds for each $x\in uBu$.
Note that $bua \in A$, since $A$ is an ideal of $B$, and
that $u=\beta_{g^{-1}}(\beta_g(u))=\beta_{g^{-1}}(ba)=\beta_{g^{-1}}(bua)$.
This shows that $u \in D_{g^{-1}}$.
For any $x\in D_{g^{-1}} \cap uBu$ we have that
$\alpha_g(x)=\beta_g(x)=bxa$.
In particular, $\alpha_g(u)=\beta_g(u)$.
Now, define $a'=ua \in u A \alpha_g(u)$ and $b'=bu \in \alpha_g(u) A u$.
It is easy to see that $a'b'=u$ and $b'a'=\alpha_g(u)$.
From the fact that $\beta_{g^{-1}}(A) \ni u$ is an ideal of $B$
we get that $uAu \subseteq uBu \subseteq D_{g^{-1}}$.
We conclude that $\alpha_g(x)=b'xa'$ holds for any $x\in uAu$.
This shows that $\alpha_g$ is inner at $u$.

We now show the contrapositive of the ''only if'' statement. Suppose that $\alpha_g$ is inner at $u$. There are elements $a \in uA\alpha_g(u)$ and $b \in \alpha_g(u)Au$, satisfying $ab=u$ and $ba=\alpha_g(u)$, such that $\alpha_g(x)=bxa$ holds for each $x\in uAu$.
Using that $\alpha$ is a restriction of $\beta$,
we know that $\alpha_g(x)=\beta_g(x)$ holds for each $x\in D_{g^{-1}}$. Note that $uAu=uBu$, since $u$ is an idempotent of $A$ which is an ideal of $B$. Hence, $uBu \subseteq D_{g^{-1}}$ and we conclude that
$\beta_g(x)=\alpha_g(x)=bxa$ holds for each $x\in uBu$.
In particular, $\beta_g(u)=\alpha_g(u)$ which makes it easy to see that $a$ and $b$ have the desired properties.
This shows that $\beta_g$ is outer at $u$.

(b): Suppose that $\alpha$ is outer.
There is a non-zero idempotent $u\in A$ such that
for each non-identity $g\in G$, the map $\alpha_g$ is outer at $u$.
It now follows immediately from (a) that, for each non-identity $g\in G$,
the map $\beta_g$ is outer at $u$. This shows that $\beta$ is outer.
For the proof of the last part,
we assume that $B$ is unital.
Seeking a contradiction, suppose that $\beta$ is not outer (in the classical sense).
Then there is a non-identity $g\in G$ such that the automorphism $\beta_g : B \to B$ is inner at $1$.
Since $u\leq 1$, Proposition \ref{proppartialorder} yields that $\beta_g$ is inner at $u$,
which is a contradiction.
\end{proof}

\begin{rem}
Note that Proposition \ref{OuterGlobalization} does not make use of
the assumption, on the existence of local units, that is made in the beginning of Section \ref{sectionpartialactions}.
\end{rem}

\begin{rem}
Note that the converse of Proposition \ref{OuterGlobalization}(b) does not hold in general.
In light of Remark \ref{rem:innerlocalglobal},
we want to
underline
that
even if $\alpha_g$, for some $g\in G$, is inner at an idempotent of $A$,
it is fully possible for the globalization $\beta$ to be outer (in the classical sense).
In fact, $\beta$ could potentially be outer at any idempotent, as long as the idempotent lies outside of $A$.
\end{rem}

\noindent
{\bf Proof of Theorem \ref{generalizationcrow}.}
The ''only if'' statement follows from 
Proposition \ref{SimpleImpliesGSimple}
and the fact that graded simplicity is a necessary condition for simplicity.
Now we show the ''if'' statement.
Suppose that $A$ is a $G$-simple ring.
Let $u$ be a non-zero idempotent of $A$ such that
for each non-identity $g \in G$,
the map $\alpha_g$ is outer at $u$.
Put $S = (u \delta_e) (\SkewRing) (u \delta_e).$
By Theorem \ref{gather}, we are done if we can show that
$Z(S)$ is a field.
Let $(u \delta_e) ( \sum_{g \in G} a_g \delta_g ) (u \delta_e)$
be a non-zero element of $Z(S)$,
where $a_g \in D_g$ is zero for all but finitely many $g \in G$.
Fix $g \in G$ so that $(u \delta_e)(a_g \delta_g)(u \delta_e) \neq 0$. 
Since $G$ is abelian,
we get that $(u \delta_e)(a_g \delta_g)(u \delta_e) \in Z(S)$.
Since $\SkewRing$ is graded simple, it is easy to see
that $S$ is also graded simple.
Therefore, the graded ideal of $S$
generated by  $(u \delta_e)(a_g \delta_g)(u \delta_e)$ 
equals $S$.
So, in particular, there is $k \in D_{g^{-1}}$
such that
\begin{equation}\label{abelian}
(u \delta_e) (a_g \delta_g) (u \delta_e) 
(k \delta_{g^{-1}}) (u \delta_e) = u \delta_e
\end{equation}
which is equivalent to the following four equivalent equations
\begin{align*}
( u a_g \delta_g ) ( u k \delta_{g^{-1}} ) (u \delta_e) = u \delta_e
&\Longleftrightarrow
(\alpha_g( \alpha_{g^{-1}}(u a_g) uk ) \delta_e) (u \delta_e) = u \delta_e \\
&\Longleftrightarrow
(u a_g \alpha_g(uk) \delta_e) (u \delta_e) = u \delta_e \\
&\Longleftrightarrow
u a_g \alpha_g(uk) u \delta_e = u \delta_e
\end{align*}
which finally gives us that
\begin{equation}\label{abelian1}
u a_g \alpha_g(uk) u = u.
\end{equation}
Note that Equation \eqref{abelian1} implies that $u \in D_g$.
Since $(u \delta_e)(a_g \delta_g)(u \delta_e) \in Z(S)$,
we can change the order of the factors on the 
left-hand side of Equation \eqref{abelian} and obtain the following three equivalent equations
\begin{align*}
(u \delta_e) (k \delta_{g^{-1}})  (u \delta_e) 
(a_g \delta_g) (u \delta_e) = u \delta_e
&\Longleftrightarrow
(uk \delta_{g^{-1}}) (u a_g \delta_g) (u \delta_e) 
= u \delta_e \\
&\Longleftrightarrow
\alpha_{g^{-1}}( \alpha_g(uk) u a_g ) \delta_e (u \delta_e) = u \delta_e
\end{align*}
which are equivalent to
\begin{equation}\label{abelian2}
\alpha_{g^{-1}}( \alpha_g(uk) u a_g ) u = u.
\end{equation}
Note that Equation \eqref{abelian2} implies that $u\in D_{g^{-1}}$,
and therefore
\begin{equation}\label{abelian3}
\alpha_g(uk) u a_g \alpha_g(u) = \alpha_g(u).
\end{equation}
Using again that $u \in D_{g^{-1}}$,
we may rewrite Equation \eqref{abelian1} and Equation \eqref{abelian3} as
\begin{equation}\label{abelian4}
u a_g \alpha_g(u) \alpha_g(u) \alpha_g(k) u = u
\end{equation}
and
\begin{equation}\label{abelian5}
\alpha_g(u) \alpha_g(k) u u a_g \alpha_g(u) = \alpha_g(u)
\end{equation}
respectively.
Furthermore, for every $b \in A$, the following three equivalent equations hold
\begin{align*}
(u \delta_e) (a_g \delta_g) (u \delta_e) (b \delta_e) (u \delta_e) =
(u \delta_e) (b \delta_e) (u \delta_e) (a_g \delta_g) (u \delta_e) \\
\Longleftrightarrow
(u a_g \delta_g) (ubu \delta_e) = 
(ubu \delta_e) ( \alpha_g( \alpha_{g^{-1}}(u a_g) u ) \delta_g ) \\
\Longleftrightarrow
\alpha_g( \alpha_{g^{-1}}(u a_g) ubu ) \delta_g  =
ubu a_g \alpha_g(u) \delta_g.
\end{align*}
The last equation yields
\begin{displaymath}
	u a_g \alpha_g(u) \alpha_g(ubu) = ubu a_g \alpha_g(u).
\end{displaymath}
By Equation \eqref{abelian5}, the last equation implies that
\begin{displaymath}
	\alpha_g(ubu) = \alpha_g(u) \alpha_g(k) u ubu u a_g \alpha_g(u)
\end{displaymath}
which shows that $\alpha_g$ is inner at $u$.
But since $\alpha_g$ is outer, at $u$, 
for non-identity $g \in G$, we conclude that $g=e$.
Hence, finally, by Equation \eqref{abelian}, we get that $Z(S)$ is a field.
\qed

\begin{rem}
We shall now make a couple of important observations.
\begin{itemize}
	\item[(a)] Outerness is not a necessary condition for simplicity of a partial skew group ring $\SkewRing$.
Indeed, consider the simple skew group ring $M_2(\R) \rtimes_\sigma \Z/2\Z$ in \cite[Example 4.1]{oinertarxiv11}.
	
	\item[(b)] Theorem \ref{generalizationcrow} does not hold for arbitrary (non-abelian) groups.
Indeed, consider \cite[Example 5.1]{oinertarxiv11}
where $X=S^1$ is the circle, $G=\Homeo(S^1)$ is the group of all homeomorphisms of $S^1$.
One may define $\sigma : G \to \Aut(C(X))$ in the usual way.
It then turns out that $C(X)$ is $G$-simple and that the action is outer. However, the skew group ring
$C(X) \rtimes_\sigma G$ is not simple.
\end{itemize}
\end{rem}

\section{An Application to Set Dynamics}\label{Sec:SetDynamics}

At the end of this section, we use Theorem \ref{generalizationcrow}
to show Theorem \ref{generalizationgoncalvesfinite}.

\begin{assumption}
Throughout this section, 
$B$ denotes a simple associative ring which has local units,
$\theta$ denotes a partial action of a group $G$
on a non-empty set $X$, and the corresponding subsets of $X$ 
are denoted by $X_g$, for $g\in G$. 
\end{assumption}

\begin{defn}\label{definitionfinitespace}
We let $F_0(X,B)$ denote the set of
functions $X \rightarrow B$ with finite support.
For each $g \in G$, let $D_g$ denote the set of 
$f \in F_0(X,B)$ such that $f(x) = 0$
for all $x \in X \setminus X_g$.
It is clear that $D_g$ is an ideal of $F_0(X,B)$
and that the map 
\begin{displaymath}
	G \ni g \mapsto 
(\alpha_g : D_{g^{-1}} \rightarrow D_g),
\end{displaymath}
defined 
by $\alpha_g(f) = f \circ \theta_{g^{-1}}$,
for $f \in D_{g^{-1}}$,
defines a partial action of $G$ on $F_0(X,B)$. 
\end{defn}

\begin{rem}
For each subset $S$ of $X$
and each 
$b \in B$, 
let $b_S$ denote the function
$X \rightarrow B$ defined by $b_S(x) = b$,
if $x \in S$, and $b_S(x)=0$, otherwise.
If $S = \{ y \}$ for some $y \in X$,
and $b \in B$, then we let $b_{ S }$ be denoted by $b_y$.
It is clear that for each $g \in G$, 
the set of $\epsilon_S$, where
$S$ is a finite subset of $X_g$ and
$\epsilon$ is a local unit in $B$ 
is a set of local units for $D_g$.
In particular, 
\begin{displaymath}
	E = \{ \epsilon_S \mid 
\mbox{$S$ is a finite subset of $X_g$ and
$\epsilon$ is a local unit in $B$} \}
\end{displaymath}
is a set of local units for $F_0(X,B)$. 
\end{rem}

For future reference we record the following result.

\begin{prop}\label{faithfulminimalimpliesfree}
If $\alpha$ is a partial action of an abelian group $G$ on a set
(Hausdorff topological space) $X$ such that $\alpha$ is faithful
and (topologically) minimal, then $\alpha$ is free.
\end{prop}

\begin{proof}
Take a non-identity $g \in G$ and consider the set
\begin{displaymath}
	F_g = \{ x \in X_{g^{-1}} \mid \alpha_g(x) = x \}.
\end{displaymath}
We need to show that $F_g$ is empty.
Take $h \in G$ and $x \in F_g \cap X_{h^{-1}}$.
By the relations (ii)-(iii) in the definition of a partial action,
and the fact that $G$ is abelian, we get that
$\alpha_h(x) = \alpha_h (\alpha_g (x) ) =
\alpha_{hg}(x) = \alpha_{gh}(x) = \alpha_g( \alpha_h(x) )$.
Thus, $F_g$ is $G$-invariant (and
closed since $X$ is Hausdorff).
Since $\alpha$ is faithful, we get that $F_g \neq X$.
Hence, we get that $F_g = \emptyset$.
Thus, $\alpha$ is free.
\end{proof}

\begin{prop}\label{minimalequivalentGsimplefinite}
$\theta$ is minimal if and only if 
$F_0(X,B)$ is $G$-simple.
\end{prop}

\begin{proof}
Suppose that $F_0(X,B)$ is not $G$-simple.
Then there is a non-trivial  
$G$-invariant ideal $I$ of $F_0(X,B)$.
Let $N_I = \bigcap_{f \in I} f^{-1}( \{ 0 \} )$.
Since $I$ is $G$-invariant
the same is true for $N_I$.
Since $I$ is non-zero, it follows that $N_I$
is a proper subset of $X$.
Seeking a contradiction, suppose that $N_I$ is empty.
Take $x \in X$ and $b \in B$.
We claim that $b_x \in I$.
If we assume that the claim holds, then,
since the set of $b_x$, for $x \in X$ and $b \in B$,
generates $F_0(X,B)$, we will get the contradiction $I = F_0(X,B)$.
Now we show the claim.
From $N_I = \emptyset$, it follows that 
there is a non-zero $c \in B$ such that $c_x \in I$.
By simplicity of $B$,
there is a natural number $n$ and
$d^{(1)},\ldots,d^{(n)},
d'^{(1)},\ldots,d'^{(n)} \in B$
such that 
$b = \sum_{i=1}^n d^{(i)} c d'^{(i)}$.
But then $b_x = \sum_{i=1}^n d^{(i)}_x c_x d'^{(i)}_x \in I$
which shows the claim.
Therefore, $N_I$ is a non-empty $G$-invariant subset of $X$, 
and hence $\theta$ is not minimal.

Now suppose that $\theta$ is not minimal.
Let $Y$ be a non-trivial $G$-invariant subset of $X$.
Let $I_Y$ denote the ideal of $F_0(X,B)$ 
consisting of all $f \in F_0(X,B)$
that vanish on $Y$. 
Since $Y$ is $G$-invariant it follows that $I_Y$ 
is $G$-invariant.
Using that $\emptyset \neq Y \neq X$, we conclude that
$I_Y$ is a non-zero proper ideal of $F_0(X,B)$.
Thus, $F_0(X,B)$ is not $G$-simple.
\end{proof}

\begin{prop}\label{faithfulequivalentinjectivefinite}
If $\alpha$ is injective, then $\theta$ is faithful.
\end{prop}

\begin{proof}
Suppose that $\theta$ is not faithful.
Then there is a non-identity $g \in G$
such that $\theta_g(x) = x$ for $x \in X_{g^{-1}}$.
This implies that $X_g = X_{g^{-1}}$
and thus that 
$D_g = D_{g^{-1}}$
and $\alpha_g (f) = f$, for $f \in D_{g^{-1}}$.
Thus, $\alpha$ is not injective.
\end{proof}

\begin{prop}\label{topologicallyfreeequivalentouterfinite}
If $\theta$ is free, 
then $\alpha$ is strongly outer.
\end{prop}

\begin{proof}
Suppose that $\alpha$ is not strongly outer.
We show that $\theta$ is not free.
Choose a non-zero idempotent $u \in F_0(X,B)$ and a non-identity $g \in G$
such that $\alpha_g$ is inner at $u$.
Pick $x \in X$ such that $b=u(x)\neq 0$. Then $b_x \leq u$
in the sense of Definition \ref{partialorder}.
By Proposition \ref{proppartialorder}, we get that
$\alpha_g$ is inner at $b_x$.
In particular, there are $f,f' \in F_0(X,B)$
such that $b_x f \alpha_g(b_x) f' b_x = b_x$,
or equivalently $b_x f b_{\theta_g(x)} f' b_x = b_x$.
Therefore we get that
\begin{displaymath}
	b_x(x) f(x) b_{\theta_g(x)}(x) f'(x) b_x(x) = b_x(x) =b \neq 0
\end{displaymath}
from which it follows that $\theta_g(x)=x$.
This shows that $\theta$ is not free.
\end{proof}

\noindent
{\bf Proof of Theorem \ref{generalizationgoncalvesfinite}.}
(i)$\Rightarrow$(iii):
Suppose that $F_0(X,B) \star_\alpha G$ is simple.
Clearly $F_0(X,B) \star_\alpha G$ is graded simple
and hence, by Proposition \ref{SimpleImpliesGSimple},
we get that $F_0(X,B)$ is $G$-simple.
By Proposition \ref{minimalequivalentGsimplefinite},
we get that $\theta$ is minimal.
By Proposition \ref{simpleimpliesinjective} we conclude that $\alpha$ is injective
and hence, by Proposition \ref{faithfulequivalentinjectivefinite}, $\theta$ is faithful.

(iii)$\Rightarrow$(ii): This follows immediately from
Proposition \ref{faithfulminimalimpliesfree}.

(ii)$\Rightarrow$(i):
Suppose that $\theta$ is minimal and free.
By Proposition \ref{minimalequivalentGsimplefinite}
and Proposition \ref{topologicallyfreeequivalentouterfinite}, we get,
respectively, that $F_0(X,B)$ is $G$-simple and that $\alpha$ is strongly outer. 
Theorem \ref{generalizationcrow} implies that $F_0(X,B) \star_\alpha G$ is simple.
\qed

\section{An Application to Topological Dynamics}\label{Sec:TopDynamics}

At the end of this section, we use Theorem \ref{generalizationcrow}
to show Theorem \ref{generalizationgoncalves}.

\begin{assumption}
Throughout this section, 
$\theta$ denotes a partial action of a group $G$
on a topological space $X$, and the corresponding subsets of 
$X$ are denoted by $X_g$, for $g\in G$.
Let $B$ denote a simple associative topological
real algebra which has a set $E$ of local units. 
Let $C_E(X,B) = \cup_{\epsilon \in E} C(X,\epsilon B \epsilon)$
where
\begin{displaymath}
	C(X,\epsilon B \epsilon) = \{
\text{continuous } f : X \rightarrow \epsilon B \epsilon \}.
\end{displaymath}
We postulate that $B$ satisfies the following property:
\begin{itemize}

\item[(P)] There is a continuous map 
$q : B \rightarrow {\Bbb R}_{\geq 0}$
satisfying $q(b) > 0$, for non-zero $b \in B$,
and $(q \circ f) \epsilon_X \in I$ for every ideal
$I$ of $C_E(X,B)$ and every $f \in I \cap C(X, \epsilon B \epsilon)$.
\end{itemize}
\end{assumption}

\begin{rem}
If $E$ and $E'$ are sets of local units 
for $B$, then $C_E(X,B) = C_{E'}(X,B)$.
In particular, if $B$ is unital, then
$C_E(X,B) = C(X,B)$ and the postulate (P)
simplifies to 
\begin{itemize}

\item[(P1)] {\it There is a continuous map 
$q : B \rightarrow {\Bbb R}_{\geq 0}$
satisfying $q(b) > 0$, for non-zero $b \in B$,
and $q \circ f \in I$ for every ideal
$I$ of $C(X,B)$ and every $f \in I$.}

\end{itemize}
\end{rem}

Now we show that there are lots of rings $B$
which satisfy the postulate (P).

\begin{exmp}\label{K}
Suppose that $K$ denotes any of the unital rings 
of real numbers ${\Bbb R}$, complex numbers ${\Bbb C}$ 
or quaternions ${\Bbb H}$
equipped with their respective conjugation
$\overline{\cdot}$, norm $| \cdot |$ and topology.
Define $q : K \rightarrow {\Bbb R}_{\geq 0}$ by
$q(k) = k \overline{k} = |k|^2$.
Then, of course, $q(k) > 0$, for non-zero $k \in K$.
If $I$ is an ideal of $C(X,K)$, then 
$q \circ I \subseteq I \overline{I} \subseteq I$
so (P1) is satisfied.
\end{exmp}

\begin{exmp}\label{matrices}
Let $K$ be defined as in Example \ref{K}.
Let $n$ denote a positive integer and let
$B$ denote the unital ring $M_n(K)$ of $n \times n$
matrices over $K$.
Extend $\overline{ \cdot  }$ to
$B$ by elementwise conjugation.
For $1 \leq i,j \leq n$, let
$e_{ij}$ denote the matrix with $1$ in
the $ij$th position and $0$ elsewhere.
For a matrix $b = (a_{ij})$ in $B$
let $q(b) = \sum_{1 \leq i,j \leq n} | a_{ij} |^2$.
It is clear that $q$ is continuous as
a map $B \rightarrow {\Bbb R }$
and that $q(b) > 0$ for non-zero $b \in B$.
Let $I$ be an ideal of $C(X,B)$ and suppose that $f \in I$.
Then for every choice of $i,j \in \{ 1,\ldots,n \}$, there 
is a continuous map $f_{ij} : X \rightarrow B$
such that $f = \sum_{1 \leq i,j \leq n} f_{ij} e_{ij}$.
Therefore, we get that 
\begin{align*}
	q \circ f &=  \sum_{1 \leq i,j \leq n} |f_{ij}|^2 = 
\sum_{1 \leq i,j \leq n} f_{ij} \overline{f}_{ij} = \\
&= \sum_{1 \leq i,j,k \leq n} e_{ki} f e_{jk} \overline{f}_{ij} \in 
\sum_{1 \leq i,j,k \leq n} e_{ki} I e_{jk} \overline{f}_{ij}
\subseteq I
\end{align*}
and hence (P1) holds.
\end{exmp}

\begin{exmp}\label{definitionB}
Let $K$ be defined as in Example \ref{K}. 
Let $B = \cup_{n \in {\Bbb N}} M_n(K)$.
Note that if $m,n \in {\Bbb N}$ satisfy 
$m \leq n$, then we may consider $M_m(K) \subseteq M_n(K)$
in the classical way.
Namely, to each $(a_{ij}) \in M_m(K)$ we associate
$(a_{ij}') \in M_n(K)$ where
$a_{ij}' = a_{ij}$, if $1 \leq i,j \leq m$,
and $a_{ij}' = 0$, otherwise.
Then $B$ is a ring which has a set of local units
$E$ consisting of the matrices $\epsilon^{(n)} = \sum_{i=1}^n e_{ii}$,
for $n \in {\Bbb N}$.
Take $b \in B$. Then $b \in M_n(K)$, for some $n \in {\Bbb N}$. 
Define $q(b)$ as in Example \ref{matrices}.
It is clear that $q(b) > 0$ if $b$ is non-zero.
Take an ideal $I$ of $C_E(X,B)$ and $f \in I \cap C(X, \epsilon^{(n)} B \epsilon^{(n)})$,
for some $n \in {\Bbb N}$.
Then $f$ belongs to $\epsilon^{(n)}_X I \epsilon^{(n)}_X$ which is an
ideal in the unital ring $C(X, \epsilon^{(n)} B \epsilon^{(n)})$.
Hence, by Example \ref{matrices}, we get that
$(q \circ f)\epsilon^{(n)}_X \in \epsilon^{(n)}_X I \epsilon^{(n)}_X \subseteq I$.
Therefore, postulate (P) holds.
\end{exmp}

\begin{defn}\label{definitionspace}
For each $g \in G$, let $D_g$ denote the set of 
$f \in C_E(X,B)$ such that $f(x) = 0$
for all $x \in X \setminus X_g$.
It is clear that $D_g$ is an ideal of $C_E(X,B)$.
\end{defn}

\begin{rem}
The set of all $\epsilon_X$, for $\epsilon \in E$, 
is a set of local units for $C_E(X,B)$.
\end{rem}

\begin{prop}
If each $X_g$, for $g \in G$, is clopen, then the map 
\begin{displaymath}
	G \ni g \mapsto (\alpha_g : D_{g^{-1}} \rightarrow D_g),
\end{displaymath}
defined by $\alpha_g(f) = f \circ \theta_{g^{-1}}$,
for $f \in D_{g^{-1}}$,
defines a partial action of $G$ on $C(X,B)$.
\end{prop}

\begin{proof}
All we need to show is that $\alpha_g$ is well-defined.
Take $f \in D_{g^{-1}}$. We need to show that
the map $h : X \rightarrow B$ defined by
$h(x) = f( \theta_{g^{-1}}(x) )$, for $x \in X_g$,
and $h(x) = 0$, for $x \in X \setminus X_g$,
is continuous. Suppose that $U$ is an open ball
in $B$. We now consider two cases.

Case 1: $0 \notin U$. Then 
$h^{-1}(U) = (f \circ \theta_{g^{-1}})^{-1}(U)$
which is open in $X_g$ and hence is open in $X$.

Case 2: $0 \in U$. Then 
$h^{-1}(U) = (f \circ \theta_{g^{-1}})^{-1}(U) 
\cup (X \setminus X_g)$ which, by Case 1 and
the fact that $X_g$ is clopen, is open in $X$.
\end{proof}

\begin{prop}\label{minimalequivalentGsimple}
If $X$ is compact Hausdorff and each $X_g$, for $g \in G$, is clopen, 
then $\theta$ is topologically minimal if and only if $C(X,B)$ is a $G$-simple ring. 
\end{prop}

\begin{proof}
Suppose that $C(X,B)$ is not $G$-simple.
There is a non-trivial  
$G$-invariant ideal $I$ of $C(X,B)$.
For a subset $J$ of $I$,
let $N_J$ be the set 
$\bigcap_{f \in J} f^{-1}( \{ 0 \} )$.
We claim that $N_I$ is a closed, non-empty
proper $G$-invariant subset of $X$.
If we assume that the claim holds
then $\theta$ is not minimal.
Now we show the claim.
Since $I$ is $G$-invariant
the same is true for $N_I$.
Since $I$ is non-zero it follows that $N_I$
is a proper subset of $X$.
Since each set $f^{-1}( \{ 0 \} )$, for $f \in I$,
is closed, the same is true for $N_I$.
Seeking a contradiction, suppose that $N_I$ is empty.
Since $X$ is compact,
there is a finite subset $J$ of $I$ such that
$N_J = N_I = \emptyset$. 
Take an arbitrary non-zero local unit $\epsilon$ in $B$.
Take another non-zero local unit $\epsilon'$ in $B$
such that $\epsilon \epsilon' = \epsilon' \epsilon = \epsilon$
and $f \in C(X,\epsilon' B \epsilon')$, for all $f \in J$.
Now define $g \in I$
by $g = \sum_{f \in J} (q \circ f) \epsilon'_X$.
Since $N_J$ is empty, we get that 
$\sum_{f \in J} (q \circ f)(x) > 0$
for all $x \in X$.
Therefore, we get that $g$ is invertible
in the ring $\epsilon'_X C(X,B) \epsilon'_X$
which in turn implies that
$\epsilon_X' \in I$.
Hence $\epsilon_X = \epsilon_X \epsilon_X' \in I$.
Since $\epsilon$ was arbitrarily chosen,
we get that $I = C(X,B)$ which is a contradiction
and therefore $N_I$ is non-empty.

Now suppose that $\theta$ is not minimal.
We show that $C(X,B)$ is not $G$-simple.
Let $Y$ be a non-trivial closed $G$-invariant subset of $X$.
Let $I_Y$ denote the ideal of $C(X,B)$ 
consisting of all $f \in C(X,B)$
that vanish on $Y$. 
Since $Y$ is $G$-invariant it follows that $I_Y$ 
is $G$-invariant.
Now we show that $I_Y$ is non-zero.
Suppose that $\epsilon$ is a non-zero local unit in $B$. 
Since $X$ is compact Hausdorff it
is completely regular. 
Hence there is a non-zero continuous
$f : X \rightarrow {\Bbb R}$ such that
$f|_Y = 0$. Define a continuous 
$\tilde{f} : X \rightarrow B$ by 
$\tilde{f}(x) = f(x) \epsilon$, for $x \in X$.
Then $\tilde{f} \in I_Y$ and therefore $I_Y \neq \{ 0 \}$.
Also, $I_Y \neq C(X,B)$. In fact, 
for every non-zero $b \in B$, the
constant function $b_X \in C(X,B) \setminus \{ I_Y \}$.
Thus, $C(X,B)$ is not $G$-simple.
\end{proof}

\begin{prop}\label{topologicallyfreeequivalentouter}
Suppose that $X$ is compact Hausdorff and
each $X_g$, for $g \in G$, is clopen.
If $\theta$ is topologically free, 
and $\epsilon \in E \setminus \{ 0 \}$,
then $\alpha$ is outer at $\epsilon_X$.
\end{prop}

\begin{proof}
Suppose that $\alpha$ is not outer at $\epsilon_X$.
We show that $\theta$ is not topologically free.
Choose a non-identity $g \in G$
such that $\alpha_g$ is inner at $\epsilon_X$.
This implies in particular that
$\epsilon_X \in D_g \cap D_{g^{-1}}$
and thus $X_g = X_{g^{-1}} = X$.
Therefore there are $f,f' \in C(X,B)$ such that 
$\epsilon_X f \alpha_g(\epsilon_X) f' \epsilon_X = \epsilon_X$ and
$\alpha_g(\epsilon_X) f' \epsilon_X f \alpha_g(\epsilon_X) = \alpha_g(\epsilon_X)$ and
$\alpha_g( \epsilon_X h \epsilon_X ) = f' h f$
for all $h \in C(X,B)$.
In particular, if we insert $h = r \epsilon_X$,
where $r \in C(X,{\Bbb R})$,
into the last equation, then we get that
$r \circ \theta_{g^{-1}} = r$
which, in turn, by Urysohn's lemma, 
implies that $\theta_g = \id_X$.
Thus, $\theta$ is not topologically free.
\end{proof}

\noindent
{\bf Proof of Theorem \ref{generalizationgoncalves}.}
(i)$\Rightarrow$(iii):
Suppose that $C_E(X,B) \star_\alpha G$ is simple.
Clearly, $C_E(X,B) \star_\alpha G$ is graded simple
and hence, by Proposition \ref{SimpleImpliesGSimple},
we get that $C_E(X,B)$ is $G$-simple.
By Proposition \ref{minimalequivalentGsimple},
we get that $\theta$ is topologically minimal.
By Proposition \ref{simpleimpliesinjective} we conclude that $\alpha$ is injective
and hence, by Proposition 
\ref{faithfulequivalentinjectivefinite}, $\theta$ is faithful.

(iii)$\Rightarrow$(ii): This follows immediately from 
Proposition \ref{faithfulminimalimpliesfree}.

(ii)$\Rightarrow$(i):
Suppose that $\theta$ is topologically minimal and topologically free.
Take any non-zero $\epsilon \in E$.
By Proposition \ref{minimalequivalentGsimple}
and Proposition \ref{topologicallyfreeequivalentouter}, we get,
respectively, that $C_E(X,B)$ is $G$-simple and that $\alpha$ is outer at $\epsilon_X$. 
Theorem \ref{generalizationcrow} implies that $C_E(X,B) \star_\alpha G$ is simple.
\qed

\proof[Acknowledgements]
The authors are grateful to an anonymous referee
for having provided them with valuable feedback and useful comments.


\begin{thebibliography}{99}


\bibitem{anh87}
P. N. \'{A}nh and L. M\'{a}rki,
{\it Morita Equivalence for Rings without Identity},
Tsukuba J. Math. {\bf 11}(1987), no. 1, 1--16.

\bibitem{AvilaFerrero11}
J. \'{A}vila and M. Ferrero,
{\it Closed and prime ideals in partial skew group rings of abelian groups},
J. Algebra Appl. {\bf 10}(2011), no. 5, 961--978.
\quad \url{http://dx.doi.org/10.1142/S0219498811005063}

\bibitem{BeuterGoncalves}
V. Beuter and D. Gon\c{c}alves,
{\it Partial crossed products as equivalence relation algebras}.
To appear in Rocky Mountain Journal of Mathematics.
{\tt arXiv:1306.3840 [math.RA]}

\bibitem{BoavaExel13}
G. Boava and R. Exel,
{\it Partial crossed product description of the C*-algebras associated with integral domains},
Proc. Amer. Math. Soc. {\bf 141}(2013), no. 7, 2439--2451.
\quad \url{http://dx.doi.org/10.1090/S0002-9939-2013-11724-7}

\bibitem{crow2005}
K. Crow,
{\it Simple Regular Skew Group Rings},
J. Algebra Appl. {\bf 4}(2005), no. 2, 127--137.
\quad \url{http://dx.doi.org/10.1142/S0219498805000909}

\bibitem{DokuchaevExel05}
M. Dokuchaev and R. Exel,
\emph{Associativity of crossed products by partial actions, enveloping actions and partial representations},
Trans. Amer. Math. Soc. {\bf 357}(2005), no. 5, 1931--1952.
\quad \url{http://dx.doi.org/10.1090/S0002-9947-04-03519-6}

\bibitem{DokuExelSimon}
M. Dokuchaev, R. Exel and J. J. Sim\'{o}n,
{\it Crossed products by twisted partial actions and graded algebras},
J. Algebra {\bf 320}(2008), no. 8, 3278--3310. 
\quad \url{http://dx.doi.org/10.1016/j.jalgebra.2008.06.023}

\bibitem{DokuchaevRio05}
M. Dokuchaev, A. Del Rio and J. J. Sim\'{o}n,
{\it Globalizations of partial actions on nonunital rings},
Proc. Amer. Math. Soc. {\bf 135}(2007), 343--352.
\quad \url{http://dx.doi.org/10.1090/S0002-9939-06-08503-0}

\bibitem{Exel94}
R. Exel,
{\it Circle actions on C*-algebras, partial automorphisms, and a generalized Pimsner-Voiculescu exact sequence},
J. Funct. Anal. {\bf 122}(1994), no. 2, 361--401.
\quad \url{http://dx.doi.org/10.1006/jfan.1994.1073}

\bibitem{Exel94BD}
R. Exel,
{\it The Bunce-Deddens algebras as crossed products by partial automorphisms},
Bol. Soc. Brasil. Mat. (N.S.) {\bf 25}(1994), no. 2, 173--179.
\quad \url{http://dx.doi.org/10.1007/BF01321306}

\bibitem{Exel95}
R. Exel,
{\it Approximately finite C*-algebras and partial automorphisms},
Math. Scand. {\bf 77}(1995), no. 2, 281--288.

\bibitem{Exel98}
R. Exel,
{\it Partial actions of groups and actions of inverse semigroups},
Proc. Amer. Math. Soc. {\bf 126}(1998), no. 12, 3481--3494.
\quad \url{http://dx.doi.org/10.1090/S0002-9939-98-04575-4}

\bibitem{ExelGiordanoGoncalves11}
R. Exel, T. Giordano and D. Goncalves,
{\it Enveloping algebras of partial actions as groupoid $C^*$-algebras},
J. Operator Theory {\bf 65}(2011), no. 1, 197--210.

\bibitem{ExelLaca99}
R. Exel and M. Laca,
{\it Cuntz-Krieger algebras for infinite matrices},
J. Reine Angew. Math. {\bf 512}(1999), 119--172.
\quad \url{http://dx.doi.org/10.1515/crll.1999.051}

\bibitem{FisherMontgomery78}
J. W. Fisher and S. Montgomery,
\emph{Semiprime skew group rings},
J. Algebra {\bf 52}(1978), no. 1, 241--247.
\quad \url{http://dx.doi.org/10.1016/0021-8693(78)90272-7}

\bibitem{Goncalves13}
D. Gon\c{c}alves,
\emph{Simplicity of partial skew group rings of abelian groups},
Canad. Math. Bull. {\bf 57}(2014), 511--519.
\quad \url{http://dx.doi.org/10.4153/CMB-2014-011-1}

\bibitem{GoncalvesOinertRoyer}
D. Gon\c{c}alves, J. \"{O}inert and D. Royer,
\emph{Simplicity of partial skew group rings with applications to Leavitt path algebras and topological dynamics},
J. Algebra {\bf 420}(2014), 201--216.
\quad \url{http://dx.doi.org/10.1016/j.jalgebra.2014.07.027}

\bibitem{haefnerdelrio}
J. Haefner and A. del Rio,
The Globalization Problem for Inner Automorphisms and Skolem-Noether Theorems,
{\it Algebras, rings and their representations}, (Proceedings of the International Conference
on Algebras, Modules and Rings,
Lisbon, Portugal, July 14-18, 2003), pp. 37--51, (2006).
\quad \url{http://dx.doi.org/10.1142/9789812774552_0005}

\bibitem{lam91}
T. Y. Lam,
{\it A First Course in Noncommutative Rings},
Springer (1991).
\quad \url{http://dx.doi.org/10.1007/978-1-4684-0406-7}

\bibitem{Mcclanahan95}
K. McClanahan,
{\it K-theory for Partial Crossed Products by Discrete Groups},
J. Funct. Anal. {\bf 130}(1995), no. 1, 77--117.
\quad \url{http://dx.doi.org/10.1006/jfan.1995.1064}

\bibitem{GroupoidGradedRings}
P. Nystedt and J. \"{O}inert,
{\it Simple Semigroup Graded Rings}.
Accepted for publication in Journal of Algebra and Its Applications.
{\tt arXiv:1308.3459 [math.RA]}

\bibitem{oinertarxiv11}
J. \"{O}inert,
{\it Simplicity of Skew Group Rings of Abelian Groups},
Comm. Algebra {\bf 42}(2014), no. 2, 831--841.
\quad \url{http://dx.doi.org/10.1080/00927872.2012.727052}

\end{thebibliography}
\end{document}